\documentclass[letter,11pt]{amsart}
\usepackage[utf8]{inputenc}
\usepackage{preamb}
\usepackage{xcolor}
\usepackage{rotating}
\usepackage{subfiles}

\usepackage[colorlinks=true,,linkcolor=blue,pagebackref, citecolor=purple]{hyperref}

\DeclareMathOperator{\CM}{\mathfrak{C}}
\DeclareMathOperator{\cok}{Coker}

\DeclareMathOperator{\MF}{MF}

\DeclareMathOperator{\spec}{Spec}

\DeclareMathOperator{\sym}{Sym}

\author{Eleonore Faber}
\address{
School of Mathematics, University of Leeds, LS2 9JT Leeds, UK
}
\email{e.m.faber@leeds.ac.uk}

\author{Colin Ingalls}
\address{
School of Mathematics and Statistics,
Carleton University, 
Ottawa, ON K1S 5B6,
Canada}
\email{cingalls@math.carleton.ca}

\author{Simon May}
\address{
School of Mathematics, University of Leeds, LS2 9JT Leeds, UK
}
\email{ll13s4m@leeds.ac.uk}

\author{Marco Talarico}
\address{
Department of Mathematics and Statistics,
University of Ottawa, 
Ottawa, ON, K1N 6N5,
Canada}
\email{mtala048@uottawa.ca}

\thanks{E.F.~was supported by EPSRC grant EP/W007509/1. C.I.~was supported by an NSERC Discovery Grant.
Simon May was supported by a EPSRC Doctoral Training Partnership (reference EP/R513258/1).}


\title{Matrix Factorizations of the  discriminant of $S_n$}
\date{\today}
\ytableausetup{centertableaux}

\begin{document}

\begin{abstract} Consider the symmetric group $S_n$ acting as a reflection group on the polynomial ring $k[x_1, \ldots, x_n]$ where $k$ is a field, such that Char$(k)$ does not divide $n!$. We use Higher Specht polynomials to construct matrix factorizations of the discriminant of this group action: these matrix factorizations are  indexed by partitions of $n$ and respect the decomposition of the coinvariant algebra into isotypical components. The maximal Cohen--Macaulay modules associated to these matrix factorizations give rise to a noncommutative resolution of the discriminant and they correspond to the nontrivial irreducible representations of $S_n$.  All our constructions are implemented in Macaulay2 and we provide several examples. We also discuss extensions of these results to Young subgroups of $S_n$. 
\end{abstract} 
\maketitle

\section{Introduction}

The classical discriminant $D(f)$ of a polynomial $f$ in one variable over a field $k$ detects whether $f$ has a multiple root. If $f$ is of degree $d$, then its discriminant can be expressed as an irreducible quasi-homogeneous polynomial in the coefficients of $f$, and $D(f)$ vanishes exactly when $f$ has a multiple root. In general, an explicit formula for $D(f)$ consists of many monomial terms (e.g., for $d=6$ the discriminant has $246$ terms), and several compact \emph{determinantal formulae} are known, that is, $D(f)$ can be written as determinant of a matrix with entries polynomials in the coefficients of $f$: the most famous determinantal formula is due to Sylvester, and there are other determinantal representations due to Bezout and Cayley, see \cite[Chapter 12, 1]{GKZ}. One can show that these matrices are equivalent in the sense that they have isomorphic cokernels, see \cite[Thm.~2.2.6]{HovinenThesis}. From a more homological point of view, making use of \emph{matrix factorizations}, these cokernels yield maximal Cohen--Macaulay (=CM)-modules of rank $1$ over the hypersurface ring defined by $D(f)$.

Now a guiding question for our investigations is: \emph{can one find other non-equivalent determinantal formulae for $D(f)$, and more generally, find other matrix factorizations of $D(f)$, and even classify them? }

In this paper we will explicitly determine several matrix factorizations of $D(f)$ that are coming from an interpretation of $D(f)$ as discriminant of the reflection group action of $S_n$ on $k^n$, in particular, our matrix factorizations will correspond to isotypical components of the coinvariant algebra. 

Before commenting on the contents of the present paper, we review some results that lead to our work. When $f$ has degree four, Hovinen studied matrix factorizations for the classical discriminant in his thesis \cite{HovinenThesis}, where he describes several non-equivalent determinantal formulae (in particular, the open swallowtail) using deformation theory and also gives a complete classification of the homogeneous rank $1$ modules of $D(f)$ \cite[Thm.~3.2.1, Thm.~4.4.7]{HovinenThesis}. In singularity theory, discriminants occur in various guises, often as so-called free divisors. Free divisors were first studied by Saito \cite{Saito80} and are hypersurfaces, whose singular locus is a CM-module over the coordinate ring. Discriminants of reflection groups have been studied from this point of view in \cite{Saito-Reflections, Arrangments}, and other discriminants include discriminants of versal deformations of several types of singularities, see \cite{BuchweitzEbelingetc} for an overview and further references.

Here we interpret the classical discriminant as the discriminant of the reflection group $S_n$ acting on $k^n$: let $G$ be any finite reflection group $G \subseteq \GL(n,k)$ acting on the vector space $k^n$. Then $G$ also acts on $S:=\mathrm{Sym}_k(k^n)$. Denote by $R:=S^{G}$ the invariant ring under the group action and further by $\mathcal{A}(G)$ the reflection arrangement in $k^n$, and by $V(\Delta)$ the discriminant in the (smooth) quotient space $k^n/G$. Note that the hypersurface $V(\Delta)$ is given by the reduced polynomial $\Delta \in R$ and is simply the projection of $\mathcal{A}(G)$ onto the quotient. Moreover, in the case of $G=S_n$ and $k=\C$, it is well-known that $V(\Delta)$ is isomorphic to the classical discriminant $V(D(f))$, where $f$ is a polynomial of degree $n$, see Section  \ref{Sub:deformdisc}.

This interpretation allows us to use representation theory, in particular the McKay correspondence (see e.g. \cite{BuchweitzMFO} for more background information and references). In \cite{BFI} a McKay correspondence was established for the discriminants $V(\Delta)$ of true reflection groups, also see \cite{BFI-ICRA} for a more leisurely account:

\begin{thm*}[Cor.~4.20 in \cite{BFI}] Let $G \subseteq \GL(n,k)$ be a true reflection group acting on $S$ and let $R=S^G$ the invariant ring,  $\Delta \in R$ be the discriminant polynomial, and $z \in S$ be the polynomial defining the reflection arrangement $\mathcal{A}(G)$. Then the nontrivial irreducible $G$-representations are in bijection with the isomorphism classes of (graded) $R/(\Delta)$-direct summands of the CM module $S/(z)$ over $R/(\Delta)$. 
\end{thm*}

Now we can state a refined version of our guiding question above: \emph{Can we write down matrix factorizations for the direct summands of $S/(z)$ explicitly and also find a geometric interpretation of them?} \\

So far, all matrix factorizations for isotypical components have been determined for the case when $G$ is a true reflection group of rank $2$, see \cite{BFI}, and for the case of the family of rank $2$ complex reflection groups $G(m,p,2)$, see \cite{May21}. For higher rank reflection groups, a complete answer is only known for the special case $S_4$ \cite[Section 6]{BFI}. There has been progress on determining the direct summands of $S/(z)$ that correspond to logarithmic (co-)residues \cite[Thm.~5.9]{BFI}. The other isotypical components have yet to be determined in general. However, in this paper we determine explicit matrix factorizations for $S/(z)$ for $G=S_n$, which may help to answer the question above. \\
The main problem in writing down the matrix factorizations is to find a suitable $R$-basis of $S/(z)$ that respects the decomposition in isotypical components. To this end we will use modifications of \emph{higher Specht polynomials}. Higher Specht polynomials themselves are a generalization of the classical Specht polynomials and were introduced by Ariki--Terasoma--Yamada \cite{ariki1997} for the groups $G(m,p,n)$, also see Terasoma--Yamada \cite{HSPSN} for the case $S_n$. These polynomials were further studied as basis for generalizations of coinvariant rings \cite{GillespieRhoades}, such as Garsia--Procesi modules (introduced in \cite{GarsiaProcesi}). They form a basis of the coinvariant algebra and are indexed by standard Young tableaux $T,P$ of shape $\lambda$, where $\lambda$ is a partition of $n$. Note that partitions of $n$ are in bijection with the irreducible representations of $S_n$.

Our main result is the following:

\begin{thm}[cf.~Theorem \ref{Big}]
Let $\Delta \in R$ be the discriminant polynomial of $G=S_n$ under the action on $S$, let $\lambda$ be a partition of $n$ and denote by $T \in \ST(\lambda)$ a standard Young tableau of shape $\lambda$. Then multiplication by $z$ on $S$ defines a matrix factorization $(z,z)$ of $\Delta$, which decomposes as
$$(z,z)= \bigoplus_{\lambda\vdash n} \bigoplus_{T\in \text{ST}(\lambda)} (z|_{M_{T}},z|_{N_{T'}}) \ , $$
where $M_T$ is the $R$-module generated by the modified higher Specht polynomials, and $N_{T'}$ is the $R$-module generated by the higher Specht modules for the conjugate tableau $T'$.
In terms of CM-modules, this decomposition can be written as
$$S/(z)=\bigoplus_{\lambda \vdash n} \bigoplus_{T \in \ST(\lambda)}M_T \cong \bigoplus_{\lambda \vdash n} \bigoplus_{T' \in \ST(\lambda)}N_{T'} \ .$$
\end{thm}

Here we note that the higher Specht polynomials $F^P_T$ for $P,T \in \ST(\lambda)$ do not yield a direct sum decomposition of the $M_T$'s on the nose, so we will define \emph{modified Higher Specht polynomials} $H^P_T$ and show that they have the desired property (i.e., form a basis), see Theorem \ref{Thm:HBasis}. 

For our computations we used the computer algebra system Macaulay2 \cite{M2}. In Section \ref{Sec:M2} the code is described in more detail and also a link to a GitHub repository is provided.

Furthermore, we follow \cite{ariki1997} and also determine a decomposition of $S/(z)$ into isotypical components corresponding to Young subgroups $S_{n_1} \times \cdots \times S_{n_m} \leqslant S_n$, where $\sum_{i=1}^mn_i=n$, using our modified higher Specht polynomials, see Theorem \ref{Prodsub}.

We are hoping to  generalize our results to  the family of complex reflection groups $G(m,p,n)$, which would give us more examples for matrix factorizations of discriminants of pseudo-reflection groups. \\

In order to get a complete answer for the question above, one also needs to consider the exceptional reflection groups of rank $\geq 2$ (for $k=\C$ these are the $15$ groups $G_{23}, \ldots, G_{37}$ in the Shephard--Todd classification). Beyond the case of the family $G(m,p,n)$, it is not quite clear how to find an equivalent of a ``Specht basis'' so we pose the 

\begin{question}
Can one find analogues for Higher Specht polynomials for all pseudo-reflection groups $G$, that is, find a basis of the coinvariant algebra $S/R_+$ which is compatible with the decomposition of $S/R_+$ into $G$-irreducible modules?
\end{question}

The paper is structured as follows: in Section \ref{Sec:Prel} we recall basics of matrix factorizations, Young diagrams and introduce $\Delta$ as discriminant of $S_n$ acting on $k^n$. In Section \ref{Sec:Decomp} we prove our main result (Theorem \ref{Big}) about the decomposition $S/(z)$ into isotypical components, using (modified) Higher Specht polynomials. We also give an explicit description of the matrix factorization for an isotypical component in Theorem \ref{Bilinear} and close the section with examples and a description of our code in Section \ref{Sec:M2}. Finally, in Section \ref{Sec:YoungSubGroups} we generalize this decomposition to Young subgroups of $S_n$.

\section{Preliminaries} \label{Sec:Prel}

\subsection{Matrix factorizations}

Matrix factorizations were introduced by Eisenbud \cite{EisenbudHOMO} to study homological properties of hypersurface rings. Here we recall the main results that will be needed later, following the expositions in \cite{CMREP,YujiCM}.

\begin{defn}\label{Def:MatrixF}Let $B$ be a commutative ring and let $f \in B$. A \emph{matrix factorization} of $f$ is a pair $(\varphi,\psi)$ of homomorphisms between free $B$-modules of the same rank $n$, with $\varphi: F \xrightarrow{} G$ and $\psi: G \xrightarrow{} F$, such that
 \[ \psi  \varphi = f \cdot1_{F} \hspace{1cm} \text{and} \hspace{1cm} \varphi \psi = f \cdot1_{G} \ . \]
 We may choose bases for $F$, $G$, and then, equivalently, $\varphi$, $\psi$ are square matrices of size $n \times n$ over $B$, such that 
  \[ \psi \cdot \varphi = f \cdot1_{B^{n}} \hspace{1cm} \text{and} \hspace{1cm} \varphi \cdot \psi = f \cdot1_{B^{n}} \ .\] 
  \end{defn}
In the following, we will always assume that $B$ is either a regular local ring or that $B$ is a graded polynomial ring.

Recall that  a \emph{morphism of matrix factorizations} $(\varphi_1,\psi_1)$ and $(\varphi_2,\psi_2)$ of $f$ is a pair of matrices $(\alpha,\beta)$ such that the following diagram commutes:

\begin{center}
\begin{tikzcd}
B^{n_1}  \arrow[d,"\alpha"] \arrow[r,"\psi_1"]& B^{n_1} \arrow[d,"\beta"] \arrow[r,"\varphi_1"] & \arrow[d,"\alpha"] B^{n_1} \\ B^{n_2} \arrow[r,"\psi_2"] & B^{n_2} \arrow[r,"\varphi_2"] & B^{n_2}     
\end{tikzcd}
\end{center}
We say that two matrix factorizations are \emph{equivalent} if there is a morphism $(\alpha,\beta)$ in which $\alpha,\beta$ are isomorphisms. Furthermore, for two matrix factorizations $(\varphi_1,\psi_1)$ and $(\varphi_2,\psi_2)$ of $f$, their \emph{sum} is defined as
\[(\varphi_1,\psi_1) \oplus (\varphi_2,\psi_2) = \left( \begin{bmatrix} \varphi_1 & 0 \\ 0 & \varphi_2\end{bmatrix},\begin{bmatrix} \psi_1 & 0 \\ 0 & \psi_2\end{bmatrix}\right) \ . \]

With these notions, matrix factorizations of $f$ form an additive category, denoted by $\MF_B(f)$.

The main reason to consider matrix factorizations is that they correspond to maximal Cohen--Macaulay modules over a hypersurface ring: For any non-unit $f \neq 0$ in $B$ we denote by $A=B/(f)$ the hypersurface ring defining $V(f) \subseteq \spec(B)$. Let further $\CM(A)$ be the \emph{category of maximal Cohen-Macaulay modules} over the ring $A$.

\begin{thm}(Eisenbud's matrix factorization theorem, see  \cite[6.1, 6.3]{EisenbudHOMO}) \\
Assume that $B$ is a regular local ring. Let $A= B/(f)$ be as above and let $(\varphi, \psi)$ be a matrix factorization of $f$. Then the functor $\cok(\varphi,\psi) = \cok(\varphi)$ induces an equivalence of categories 
$$\underline{MF}_B(f):=MF_B(f)/\{(1,f)\} \simeq \CM(A) \ . $$ \end{thm}

This shows that instead of  directly calculating the maximal Cohen-Macaulay modules over $A$, we can instead construct matrix factorizations of $f$.

\begin{remark} This theorem also holds in the \emph{graded case}, that is, when $B$ is a graded polynomial ring and $f$ is a homogeneous element. Then one considers the categories of graded matrix factorizations and of graded CM-modules, see e.g., \cite{YujiCM}. In this paper we implicitly work in the graded situation, although we will not care too much about the actual degrees.
\end{remark}

\subsection{Young Diagrams}\label{YD}
Here we recall basic facts about Young diagrams and representations of $S_n$, for more detail see \cite{fulton_1996}.

Consider $n \in \mathbb{N}\setminus \{0\}$. Let $\lambda$ be a partition of $n$, denoted $\lambda \vdash n$, i.e $\lambda =(\lambda_0,\ldots,\lambda_{k-1})$, such that $1 \leq k\leq n$, $0 < \lambda_{i+1}\leq \lambda_i$ and $\sum_{i}n_i=n$. 
A partition can also be represented as a \emph{Young diagram}, which is constructed in the following way: Given a partition $\lambda$ of $n$, the Young diagram associated to $\lambda$ is a collection of left justified rows of squares called cells. Enumerate the rows from $0$ to $k-1$, top to bottom, the number of cells in row $i$ is $\lambda_{i}$. The partitions uniquely determine the Young diagram so we use the same notation $\lambda$ for the partition and the Young diagram. We call a Young diagram associated to a partition of $n$, a Young diagram of size $n$.

\begin{example}

Let $n=5$ and $\lambda=(2,2,1)$ then the Young diagram is: 

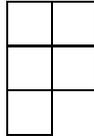
\[\ydiagram{2,2,1}\]
\end{example}

\begin{defn}
A \textit{Young tableau} is a Young diagram of size $n$, where each cell contains a number from $1$ to $n$ such that each number $1$ to $n$ appears only once. A Young tableau on with an underlying Young diagram $\lambda$ is said to be of \emph{shape} $\lambda$. A Young tableau is called \textit{standard} the sequence of entries in the rows and columns are strictly increasing, and the set of standard Young tableau of shape $\lambda$ is $\ST(\lambda)$. \\
For a Young tableau $T \in \ST(\lambda)$ we write $T'$ for its \emph{conjugate tableau}, that is, $T'$ is obtained by transposing $T$ and its entries. Note that $T' \in \ST(\lambda')$, where $\lambda'$ is the conjugate partition of $\lambda$.
\end{defn}

\begin{example}

Let $n=5$ and consider the Young diagram $\lambda$ from the previous example. Then the following are Young tableau:

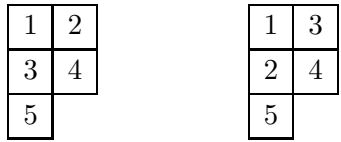
\[\begin{ytableau}
1&2 \\ 3& 4 \\ 5 
\end{ytableau}\hspace{2cm} \begin{ytableau}
1&3 \\ 2& 4 \\ 5
\end{ytableau}\]

These are also both standard tableau.
\end{example}
\begin{defn} \label{def:index}
Let $\lambda\vdash n$ and let $T$ be a standard tableau of shape $\lambda$. We define the \emph{word} $w(T)$ to be the sequence obtained by reading each column from bottom to top starting from the left. We write $w(T)^{i}$ for the $i$-th term in this sequence, where $i=0, \ldots, n-1$. The \emph{index} $i(T)=i(w(T))$ is inductively defined as; $i(1)=0$, if $i(k)=p$ then $i(k+1)=p$ if $k+1$ is to the right of $k$ in $w(T)$ or $i(k+1)=p+1$ if  $k+1$ is to the right of $k$ in $w(T)$. We write $i(T)$ as a tableau with the indexes in the corresponding cells. Further we define $\hat{i}(T)$ to be $i(T)$ written in non decreasing order and $|i(T)|$ to be the sum of the indexes. This notion will be needed in Section \ref{Sec:Decomp}, in particular Lemma \ref{Lem:Bilinear}.
\end{defn}

\begin{example}
\[T=\begin{ytableau}
1&2 \\ 3& 4 \\ 5 
\end{ytableau} \hspace{2cm} i(T)= \begin{ytableau}
0 &0 \\ 1& 1 \\ 2 
\end{ytableau}\]
\[\hat{i}(T)=(0,0,1,1,2) \hspace{2cm} |i(T)|=4\]
\end{example}

It is widely known that Young diagrams of size $n$, and thus partitions of $n$, are in bijection with the irreducible representations of $S_n$, where Char$(k)$ does not divide $|S_n|=n!$, see \cite[Section 4]{fultonrep} for Char$(k)=0$ and \cite[Section 10, 11]{JamesRep} for Char$(k) \nmid n!$ . We will sometimes denote the irreducible representations $V_\lambda$ of $S_n$ simply by their corresponding partitions $\lambda$.

\subsection{The action of \(S_n\) on \(k^n\)}

$S_n$ naturally acts on a finite dimensional vector  space $V$ of dimension $n$ over field $k$ where Char$(k)$ does not divide $|S_n|=n!$. The quotient variety $V/S_n$ is smooth by the theorem of Chevalley--Shephard--Todd \cite{Che55}. By fixing a basis $\{x_1,\ldots,x_n\}$ of $V$, we form the symmetric algebra of $V$,  $\sym_{k}(V) \cong k[x_1,\ldots,x_n]$. The action of $S_n$ on $V$ can be naturally extended to $S=k[x_1,\ldots,x_n]$ via $\pi \cdot f(x)=f(\pi(x))$ for $\pi \in S_n$. We denote by $R$ the invariant ring $R=S^{S_n}$. Note that we have $\spec(S)=V$ and $\spec(R)=V/S_n$. The theorem of Chevalley--Shepard--Todd also shows that $R\cong k[e_1,\ldots,e_n]$, where $e_i$ are the elementary symmetric polynomials in $x_1,\ldots,x_n$. Note that $R$ is a graded polynomial ring with $\deg e_i=i$.

Let $\mathcal{A}(S_n)$ be the set of reflecting hyperplanes of the action of $S_n$, the so-called reflection arrangement of $S_n$. Let $H\in \mathcal{A}(S_n)$ be such a hyperplane, and let $\alpha_H$ be a linear form defining $H$ in $S$.  Then
\[z= {\displaystyle \prod_{H\in \mathcal{A}(S_n)}^{} \alpha_H}=\prod_{1 \leq i < j \leq n} (x_i-x_j)\]
is a polynomial in $S$ defining the reflection arrangement of $S_n$, that is $V(z)= \bigcup_{H \in \mathcal{A}(S_n)} H$.

\begin{defn}
The discriminant polynomial of the $S_n$ action on $V$ is defined by:
\[\Delta= z^2 = {\displaystyle \prod_{H\in \mathcal{A}(S_n)}^{} \alpha_H^2}=\prod_{1 \leq i < j \leq n} (x_i-x_j)^2 \ . \]

This is defined as an element of $S$ but $\Delta$ is also invariant under the group action, see \cite[Lemma 6.44]{Arrangments}, and so can be expressed as an element of $R=k[e_1,\ldots,e_n]$. We note that $R/(\Delta)$ is a hypersurface ring and from Chevalley's theorem, $S$ is a free $R$-module of rank $|S_n|=n!$. 

Let $(R_+)$ be the ideal generated by $\sigma_1,\ldots ,\sigma_n$ in $S$ and let $S/(R_+)$ the \emph{coinvariant algebra}. The structure of $S$ as a graded free $R$-module is given by Chevalley's theorem \cite{Che55}, Chevalley assumes field of characteristic $0$ but the result holds more generally for characteristic $k$ not dividing $|S_n|$, see  \cite[Chapter 5, Section 2, Theorem 2]{frenchChev}. As a graded $R$-module $S$ can be decomposed as:
\[S\cong S/(R_+)\otimes_{k} R\]

Denote the set of irreducible representations $V_\lambda$ of $S_n$ by $\mathrm{irrep}(S_n)$. The $R$-module $S/(R_+)$ carries the regular representation, in particular:

\[S/(R_+)\cong\bigoplus_{V_\lambda \in \text{irrep}(S_n)}V_\lambda^{\dim V_\lambda} \ . \]

We thus denote the $\lambda$-direct summand (the \emph{$\lambda$-isotypical component}) of $S$ by $S_\lambda=V_\lambda^{\dim V_\lambda} \otimes_k R$. 
The polynomial $z$ is the relative invariant for the determinantal representation, see  \cite[Theorem 6.37]{Arrangments} (This was proved in \cite[Theorem 3.1]{stanley} for characteristic $0$). That is, $z$ generates the direct summand of $S/(R_+)$ corresponding to $V_{\det}=V_{\lambda}$ where $\lambda$ is given by the Young diagram

\[T=\ytableausetup{mathmode, boxsize=1.5em}
\begin{ytableau}
\,\\
\,\\
\none[\scriptstyle \vdots]
\\
\, \\
\,
\end{ytableau}\]

Note that if $a \in S_{\lambda}$, then $za \in S_{\lambda \otimes \det}$.
In the following we always denote $\lambda \otimes \det$ by $\lambda'$ and note that $\lambda'$ corresponds to the conjugate representation $V_{\lambda'}$ of $V_\lambda$. 

Recalling Definition \ref{Def:MatrixF}, we have that multiplication by $z$ induces the matrix factorization  $(z,z)$ over $R$ of $\Delta$: 
\[
\begin{tikzcd}
S  \arrow[d,"\begin{rotate}{-90}$\cong$\end{rotate}"] \arrow[r,"z"]& S \arrow[d,"\begin{rotate}{-90}$\cong$\end{rotate}"] \arrow[r,"z"] & \arrow[d,"\begin{rotate}{-90}$\cong$\end{rotate}"] S \\ 
R^{n!} \arrow[r,"z"] & R^{n!} \arrow[r,"z"] & R^{n!}     
\end{tikzcd} \] 
This matrix factorization of $\Delta \in R$ corresponds to the maximal Cohen--Macaulay module $\cok(z)=S/(z)$. It was shown in \cite{BFI} that  $\End_{R/(\Delta)}(S/(z))$ has global dimension $n$ and $S/(z)$ is a faithful $R/(\Delta)$ module and so is a \emph{noncommutative resolution} of the discriminant and that the direct summands of $S/(z)$ correspond to the nontrivial irreducible representations of $S_n$. 

\end{defn}

\subsection{Discriminants of reflection groups and discriminants of deformations (over $k=\C$)} \label{Sub:deformdisc}

Here we briefly comment on the connection between the classical discriminant of a polynomial (as discussed in the introduction) and discriminants of reflection groups: let $k=\C$ and let $G\subseteq \GL(n,k)$ be a finite complexified Coxeter group. That is, $G$ is of type $A_k, B_k, D_k, E_6, E_7, E_8, I_2(p), F_4, H_3$, or $H_4$, see e.g.~\cite{Humphreys} for the classification. Then Arnol'd has shown that the discriminant of the reflection group $G$ in $\C^n/G$ is isomorphic to the discriminant of a semi-universal deformation of the singularity of the same type, see \cite{Arnold} for type ADE, and \cite{ArnoldVarchenkoGuseinZadebook1} for more details. Since $S_n$ in its reflection representation corresponds to the Coxeter group $A_{n-1}$, our discriminant $V(\Delta)$ is isomorphic to the discriminant of the semi-universal deformation of an $A_{n-1}$-singularity. A semi-universal deformation of the singularity $k[x]/(x^n)$ is given by
$$F=x^n + a_{n-2}x^{n-2}+ \cdots + a_1x +a_0 \ .$$
The discriminant of $F$ is the classical discriminant of a polynomial of degree $n$.     

 For a concrete example, look at the correspondence for $n=3$: Consider a cubic monic polynomial $f(x)=x^3+ax^2+bx+c$ with $a,b,c \in k$. Using Sylvester's formula, one calculates that the discriminant $D(f)$ is given as
$$D(f)(a,b,c)=-4a^3c + a^2b^2 + 18abc - 4b^3 - 27c^2 \ .$$
$D(f)$ is a quasi-homogeneous polynomial in $k[a,b,c]$ with $\deg(a)=1, \deg(b)=2, \deg(c)=3$. Moreover, one can always achieve $a=0$, and then the discriminant is of the well-known form
$$D(f)(b,c)=4b^3+27c^2 \ . $$
On the other hand, we calculate the discriminant $\Delta$ of the action of $S_3$ on $k^3$ resp. $k[x_1,x_2,x_3]$ from its Saito matrix (see \cite{Saito-Reflections} for Coxeter groups and \cite{Arrangments} for complex reflection groups): the Saito matrix is given as $JJ^T$, where $J$ is the Jacobian matrix of the basic invariants. In the case of $S_3$, we can take the power sums $s_i=\sum_{j=1}^3x_j^i$, $i=1,2,3$ for the basic invariants and then $J=(\frac{\partial s_i}{\partial x_j})_{i,j=1,\ldots, 3}$ (up to multiplication with a constant)
$$JJ^T=\begin{pmatrix} 1 & 1& 1 \\ x_1 & x_2 & x_3 \\ x_1^2 & x_2^2 & x_3^2 \end{pmatrix}\begin{pmatrix} 1 & x_1 & x_1^2 \\ 1 & x_2 & x_2^2 \\ 1 & x_3 & x_3^2 \end{pmatrix} = \begin{pmatrix} 3 & s_1 & s_2 \\ s_1 & s_2 & s_3 \\ s_2 & s_3 & s_4 \end{pmatrix} \ . $$
One calculates $s_4=\frac{1}{6}(s_1^4-6s_2s_1^2+3s_2^2+8s_3s_1)$ and further 
$$\Delta=\det(JJ^T)=3s_2s_1^4 - 7s_1^2s_2^2 + 12s_1s_2s_3 + s_2^3 - 6s_3^2 - 1/3s_1^6 - 8/3s_3s_1^3 \ .$$
A coordinate change shows (and restricting to the invariant hyperplane $s_1=x_1+x_2+x_3=0$) shows that this defines the same curve as $D(f)$, namely the cusp $\Delta=4s_2^3+27s_3^2$ in $R=k[x_1,x_2,x_3]^{S_3}\cong k[s_2,s_3]$. 

\section{Decomposition of $(z,z)$ for $S_n$}\label{Sec:Decomp}

During this section fix $n \geq 3$. We consider the decomposition of the coinvariant algebra $S/(R_+)$ and the multiplication of $z$ restricted to each isotypical component $S_\lambda$, where each $\lambda$ corresponds to a Young tableau. Basis elements for the isotypical components $S_\lambda$ are then given by Higher Specht polynomials \cite{ariki1997}, we follow the definitions as in loc.~cit. However, for our purposes we will define a modification of these polynomials, see Definition \ref{def:HSP}.

\begin{defn}
Let $\lambda$ $\vdash n$ and  $T_1,T_2 \in \ST(\lambda)$. We define the \emph{Last Letter Ordering (LL)} in the following way. Let $1\leq k \leq n$ be the largest integer that is written in a different position for both tableaux $T_1$ and $T_2$. If the row in which $k$ appears in $T_2$ is above the row it appears in $T_1$, then we say $T_1<T_2$.
\end{defn}

\begin{example}

Let $n=5$, Consider the following two tableaux $T_1$ and $T_2$ on the partition $(3,2)$.

\[T_1=\begin{ytableau}
    1 & 2 & 4 \\
    3 & 5\\
    \end{ytableau}
     <
    \begin{ytableau}
    1 & 3 & 4\\
    2 & 5\\
    \end{ytableau}=T_2\]

The set of elements that are in different positions is $\{2,3\}$ thus the maximal element which has a different position is $3$. Note that in $T_1$ the element 3 is written in the second row, which is below the first row where $3$ is written in $T_2$.

\end{example}
\begin{defn}
Given a Young tableaux $T$ of shape $\lambda$, we define two subgroups of $S_n$, the \emph{Row Stabilizer} $R(T)$ which are all elements of the group ring $kS_n$ that permute elements within the same row, and similarly the \emph{Column Stabilizer} $C(T)$ which permutes elements within the same columns of $T$. With these subgroups we define the following
\begin{align*}
r_T = \sum_{\pi \in R(T)}\pi &&c_T = \sum_{\rho \in C(T)}\sgn(\rho) \rho \ .
\end{align*}
Lastly we define the \emph{Young Symmetrizers}
\[\e_T = \frac{f^\lambda}{n!}c_Tr_T \hspace{1cm} \text{and} \hspace{1cm} \s_T = \frac{f^\lambda}{n!}r_Tc_T\] where $f^\lambda$ is the number of standard tableau of shape $\lambda$. These are both idempotents of $kS_n$.
\end{defn}

 In the following we  use multi-index notation $x_{w(T)}^{i(P)}=x_{w(T)^0}^{w(i(P))^0}\cdots x_{w(T)^{n-1}}^{w(i(P))^{n-1}}$.
    
    \begin{defn} \label{def:HSP}
    Let $T,P$ be two standard Young tableaux of shape $\lambda$. The \emph{higher Specht polynomials} are defined as
    \[
        F_T^P = \e_T.x_T^P  = \e_T.x_{w(T)}^{i(P)}\ ,
    \]
    We further define the \emph{modified higher Specht polynomials} as 
    $$ H_T^P = \s_T.x_T^P \ . $$
    
    \end{defn}

\medskip

\begin{defn}
For a tableaux $ P \in \ST(\lambda)$ let $M^P$ the $R$-submodule of $S/(R_+)$  generated by $\{H_T^P \st T \in \ST(\lambda)\}$ and $N^P$ the $R$-submodule of $S/(R_+)$ generated by $\{F_T^P \,|\, T \in \ST(\lambda)\}$.
\end{defn}

\begin{thm}\cite[Theorem 1, (2)]{ariki1997}
Let $P \in \ST(\lambda)$, then $M^P$ is a $S_n$-subrepresentation of $S$ isomorphic to the irreducible representation corresponding to $P$. 
\end{thm}

\begin{thm} \label{Thm:HBasis}
Let $P \in \ST(\lambda)$, then $N^P$ is a $S_n$-subrepresentation of $S$ isomorphic to $M^P$.
\end{thm}
\begin{proof}
Recall that $kS_n\e_{T}$ and $kS_n\s_T$ are both isomorphic to the irreducible representation $V_\lambda$ \cite[Exercise 4.4]{fultonrep}. Therefore if we consider an ordering $\{T_1,\ldots,T_k\}$ of the standard tableaux of shape $\lambda$ according to the last letter ordering, and fix $\pi_i \in  S_n$ such that $\pi_i(T_1) = T_i$. It can be shown that $\pi_i\e_{T_i}$ and $\pi_i\s_{T_i}$ are a basis for $V_\lambda$ in $kS_n$ for this result see \cite[Lemma 5]{ariki1997}. Thus consider the isomorphism $\varphi:kS_n\e_T \to kS_n\s_T$. Then any element $f \in N^P= \langle F_T^P \st T \in \ST(\lambda)\rangle $ can be written as $(c_1\pi_1\e_{T_1} + \cdots + c_k\pi_k\e_{T_k}).x_T^P $. Consider this map between $N^P$ and $M^P$ 
\[ (c_1\pi_1\e_{T_1} + \cdots + c_k\pi_k\e_{T_k}).x_T^P \longmapsto (c_1\varphi(\pi_1\e_{T_1}) + \cdots + c_k\varphi(\pi_k\e_{T_k})).x_T^P  \]
The above map is an isomorphism since $\varphi$ is an isomorphism.
\end{proof}

\begin{defn}\label{def:freemods}
For a tableau $ T \in \ST(\lambda)$ let $M_T$ the $R$-submodule of $S/(R_+)$ generated by $\{H_T^P \st P \in \ST(\lambda)\}$ and $N_T$ the $R$-submodule generated by $\{F_T^P \,|\, P \in \ST(\lambda)\}$.
\end{defn}

\begin{remark}
The modules $M_S$ and $N_T$ are not irreducible representations of $S_n$ and are free $R$-modules.
\end{remark}

\begin{thm} \cite[Theorem 1]{HSPSN} \label{Thm:TerasomaYamada}The collection

\[\bigcup_{\lambda\vdash n } \{F^{S}_{T} \st T \in \ST(\lambda), S \in \ST(\lambda)\} \] 

form a $k$-basis for $S/(R_+)$.

\end{thm}

\begin{example}
Let $T$ be the following standard tableau:

\[T=\ytableausetup{mathmode, boxsize=2em}
\begin{ytableau}
1 \\
2 \\
\none[\vdots]
\\
\scriptstyle n-1 \\
n
\end{ytableau}\] 
Note that $T$ gives rise to the determinantal representation $V_{\det}$ of $S_n$.
The Young Symmetrizer $\varepsilon_T$ is given by:
\[\varepsilon_T=\frac{1}{n!}c_Tr_T=\frac{1}{n!}\left(\sum_{\pi \in C(T)} \sgn(\pi)\pi\right)id=\frac{1}{n!}\sum_{\pi \in S_n} \sgn(\pi)\pi. \]

The index $i(T)=(n-1,n-2,\dots,1,0)$ and so $F_T^T= \varepsilon_T(x_1^0x_2^1\cdots x_{n}^{n-1})$. The higher Specht polynomial $F_T^T$ is the polynomial $\frac{1}{n!}z$. Moreover, in this case we also have 

\[\sigma_T = \frac{1}{n!}r_Tc_T =\frac{1}{n!}id\left(\sum_{\pi \in C(T)} \sgn(\pi)\pi\right)=\frac{1}{n!}\varepsilon_T.\]
 and so $H^T_T=F^T_T= \frac{1}{n!}z$
\end{example}

\begin{lem}\label{Zeroing}
If $T_1<T_2$ then $\e_{T_1}\e_{T_2}=\sigma_{T_2}\sigma_{T_1}= 0$.
\end{lem}
\begin{proof}
The proof for $\varepsilon_{T_1}\varepsilon_{T_2}=0$ is widely known, see \cite[Lemma 4]{ariki1997} or \cite[Proposition 1]{orth}. The equality $\sigma_{T_2}\sigma_{T_1}=0$ can be seen by using a similar proof.
\end{proof}

\begin{lem}
Let $T$ be a Young tableau of shape $\lambda$ and $T'$ its conjugate, then we have
\[\varepsilon_{T}(zf)=z\sigma_{T'}(f)\]
for any polynomial $f \in S$.
\end{lem}
\begin{proof}
	\sloppy{We first observe that for a Young tableau $T$ of shape $\lambda$,  $R(T)=C(T')$, $C(T)=R(T')$ and so 
	\[ \varepsilon_{T}=\sum_{r \in R(T), c \in C(T)} \sgn(c)rc= \sum_{c \in C(T'), r \in R(T')} \sgn(r)cr.\]
	We also have that for any $\pi \in S_n$, $\pi(z)=\sgn(\pi) z$, and so for any polynomial $f$;}
	\begin{equation*}
		\begin{split}
		\varepsilon_{T}(zf)&=\sum_{r \in R(T), c \in C(T)} \sgn(c)rc(zf)= \sum_{c \in C(T'), r \in R(T')} \sgn(r)cr(zf) \\ &= z\left(\sum_{c \in C(T'), r \in R(T')} \sgn(c)cr(f))= z(\sigma_{T'}(f))\right)
		\end{split}
	\end{equation*}
\end{proof}

\begin{thm}\label{Big}
For the discriminant $\Delta$ of $S_n$, the matrix factorization defined by the reduced hyperplane arrangement, $(z,z)$, can be decomposed in the following way:

\[(z,z) = \bigoplus_{\lambda\vdash n} \bigoplus_{T\in \text{ST}(\lambda)} (z|_{M_{T}},z|_{N_{T'}}).\]

$(z|_{M_{T}},z|_{N_{T'}})$ are the matrix factorizations:

  \begin{center}
    \begin{tikzcd}
    M_{T} \arrow[r, "z|_{M_{T}}"] & N_{T'} \arrow[r, "z|_{N_{T'}}"] & M_{T}
    \end{tikzcd}
    \end{center}
    $M_T$ the $R$-submodule of $S/(R_+)$ generated by $\{H_T^P \st P \in \ST(\lambda)\}$ and $N_{T'}$ the $R$-submodule generated by $\{F_{T'}^{P'} \,|\, {P'} \in \ST(\lambda')\}$.
\end{thm}

\begin{proof}
For an irreducible representation $\lambda$ of $G$, recall that the map $z$ takes elements from the isotypical component $S_{\lambda}$ of $S$ to the isotypical component $S_{\lambda'}$ of $S$ for $\lambda'=\lambda\otimes \det$. Thus the matrix factorization decomposes immediately as:
\[(z,z) = \bigoplus_{\lambda\vdash n} (z|_{S_{\lambda}},z|_{S_{\lambda'}}).\]
Let $T,P \in \ST(\lm)$, and consider $H_T^P \in S_{\lm}$ from above. Hence $zH_T^P \in S_{\lm'}$. We can write $zH_T^P$ as the following; 
    \begin{align} \label{equation1}
        zH_T^P = \sum\limits_{U,W \in \ST(\lm')}g_{U,T}^{P,W}F_U^W \ ,
    \end{align} 
    
where $g_{U,T}^{P,W}$ are in $R$, since the $F_U^W$ form an $R$-basis of $S_{\lambda'}$ by Theorem \ref{Thm:TerasomaYamada}. Recall that $F_T^P = \varepsilon_{T}.x_T^P$, and $H_T^P = \sigma_{T}.x_T^P$.  If $T_1 <T_2$ then by Lemma \ref{Zeroing} we have that $\varepsilon_{T_1}\varepsilon_{T_2} = \sigma_{T_2}\sigma_{T_1} = 0 $ and hence $\varepsilon_{T_1}F_{T_2}^P= \varepsilon_{T_1}\varepsilon_{T_2}.x_{T_2}^P = 0$ and $  \sigma_{T_2}H^P_{T_1} =\sigma_{T_2}\sigma_{T_1}.x_{T_1}^P = 0$.
    Order $\ST(\lm') = (T_1',\hdots,T_k')$, such a way that if $i<j$ then $T_i' < T_j'$. We want to calculate $zH_{T_1}^P$. Applying $\varepsilon_{T'_1}$ to both sides of equation \eqref{equation1} yields

    \begin{align*}
        \varepsilon_{T'_1}\sum\limits_{U,W \in \ST(\lm')}g_{U,T_1}^{P,W}F_U^W &=  \sum\limits_{U,W \in \ST(\lm')}g_{U,T_1}^{P,W}(\varepsilon_{T'_1}F_U^W)\\
                                  &=\sum\limits_{W \in \ST(\lm')}g_{T'_1,T_1}^{P,W}F_{T'_1}^W \ ,
    \end{align*}
    since $T_1'$ is the least element in $\ST(\lambda')$. Further compute
    \begin{align*}
        \varepsilon_{T'_1}(zH_{T_1}^P) &= z(\sigma_{T_1} H_{T_1}^P)\\
                                  &= zH_{T_1}^P \ .\\
    \end{align*}

    Thus we have that 
    
    \[zH_{T_1}^P = \sum\limits_{W \in \ST(\lm')}g_{T'_1,T_1}^{P,W}F_{T'_1}^W \ . \]
    With the argument above and the fact that for any $T_i \in \ST(\lm)$ there exists a permutation $\pi$ that permutes $T_i$ and $T_1$ we have the following equation for any $T_i \in \ST(\lm)$.

    \[zH_{T_i}^{P} =  \sum\limits_{W \in \ST(\lm)}(\sgn(\pi) g_{T_i',T_1}^{P,W})F_{T_i'}^W  \]
    
    This shows that for any $H_T^P \in S_{\lm}$ we have that $zH_T^P \in \gen{F_{T'}^P \st P \in \ST(\lm')}$. Following a similar argument above we have that when we restrict $z$ to $\gen{F_{T'}^P \st P \in \ST(\lm')}$ we have an element in $\gen{H_{T}^P \st P \in \ST(\lm)}$. 
    \[ zF_{T_i'}^{P} =  \sum\limits_{W \in \ST(\lm)}(\sgn(\pi) h_{T_i,T'_k}^{P,W})H_{T_i}^W \]
    We can see that for a $T \in \ST(\lm)$ and using the notation above we can write the matrices of the maps as 
        \[ \left[g_{T,T'}^{P,W}\right]\left[h_{T',T}^{P,W}\right]_{P,W'\in \ST(\lm)} = \Delta \text{Id}_{|\ST(\lambda)|\times |\ST(\lambda)|}  \]

\end{proof}
\bigskip

\begin{thm} \label{Big-equiv}
If $T_1,T_2 \in \ST(\lambda)$ then there is a matrix factorization equivalence between $(z|_{N_{T_1}},z|_{M_{T_1'}})$ and $(z|_{N_{T_2}},z|_{M_{T_2'}})$.
\end{thm}
\begin{proof}
Consider $\pi \in S_n$ such that $\pi(T_1) = T_2$, thus for any $P \in \ST(\lambda)$ we have that 
\[zH_{T_1}^P = \sgn{(\pi)}zH_{T_2}^P \:\text{ and }\: zF_{T_1'}^{P'} = \sgn{(\pi)}zF_{T_2'}^{P'}\] 
This means that $z|_{N_{T_1}} = \sgn(\pi)z|_{N_{T_2}}$ and $z|_{M_{T_1'}} = \sgn(\pi)z|_{M_{T_2'}}$. Therefore the matrices are the same up to multiplication by a scalar matrix there is a matrix factorization equivalence between them. 
\end{proof}

\begin{remark}
Let $\lambda$ be a Young diagram of size $n$, then $|\ST(\lambda)|= \dim S_\lambda$, so we get $\dim S_\lambda$ copies of the maximal Cohen--Macaulay-module $\cok(z|_{N_{T_1}},z|_{M_{T_1'}})$ in the decomposition.
\end{remark}

\begin{defn}
Define a $R$-bilinear form $\gen{-,-}: S\times S \to R$ where for any $f,g \in S$ we have that $ \gen{f,g} = \frac{1}{z} \sum_{\pi \in S_n} \sgn(\pi)\pi(fg)$.
\end{defn}

Note that this bilinear form is the as the one used in \cite{ariki1997}, except we do not set the variables to $0$. Recall from Definition \ref{def:index} that for a Tableau $T$, $\hat{i}(T)$ is $i(T)$ written in non decreasing order and $|i(T)|$ is the sum of the indexes. Consider the ordering on $\ST(\lambda)$, where $S_1<S_2$ if and only if $|\hat{i}(S_1)|<|\hat{i}(S_2)|$, if $|\hat{i}(S_1)|=|\hat{i}(S_2)|$ then $\hat{i}(S_1)<\hat{i}(S_2)$ with respect to the reverse lexicographical ordering, if $\hat{i}(S_1)=\hat{i}(S_2)$ then $S_1<S_2$ with respect to the last letter order. 

\begin{lem}\label{Lem:Bilinear}
Let $S_1<S_2$ with respect to the ordering above, then $\gen{F_T^{S_1},F_{T'}^{S_2'}}=0$ 
\end{lem}

\begin{proof}
The main idea here is that if $\deg(fg)<\deg(z)=\frac{n(n-1)}{2}$ then $\gen{f,g}$ is either 0 and if $\deg(fg)=\frac{n(n-1)}{2}$ then it is a constant. In these cases the result \cite[Proposition 1]{ariki1997} for $\gen{-,-}$ hold, thus it is sufficient to show that if $S_1<S_2$ then $\deg(F_T^{S_1}F_T'^{S_2'})< \frac{n(n-1)}{2}$. The Lemma then follows from the case distinction:
\begin{enumerate}
    \item If $|\hat{i}(S_1)|<|\hat{i}(S_2)|$, then $|\hat{i}(S_1)|+|\hat{j}(S_2)|<\frac{n(n-1)}{2}$ by \cite[Lemma 1]{ariki1997}.
    \item  If $|\hat{i}(S_1)|=|\hat{i}(S_2)|$ and $\hat{i}(S_1)<\hat{i}(S_2)$ in the reverse lexicographical ordering, then $|\hat{i}(S_1)|=\frac{n(n-1)}{2}-|\hat{j}(S_2)|$ and $|\hat{i}(S_1)|+|\hat{j}(S_2)|=\frac{n(n-1)}{2}$ and thus from the proof of \cite[Theorem 1]{ariki1997} the results hold. 
    \item If $|\hat{i}(S_1)|=|\hat{i}(S_2)|$ and $\hat{i}(S_1)=\hat{i}(S_2)$, then if $S_1<S_2$ with respect to the last letter order the results hold using the same argument as in $(2)$.
\end{enumerate}
\end{proof}

\begin{thm}\label{Bilinear}
 Let $\lambda$ be a Young diagram, where $m$ is the dimension of the corresponding irreducible representation and $T \in \ST(\lambda)$. The matrix factorization $(z|_{N_{T}},z|_{M_{T'}})$ can be written explicitly as $(A,B)$, where
    \[ A=  \begin{bmatrix}
      g_1^1& \cdots & g_m^1 \\
      \vdots & \ddots & \vdots \\
     g_1^m & \cdots & g_m^m
      \end{bmatrix} \]

$B$ is a $m\times m$ matrix obtained by taking the first reduced syzygy of $A$, $\mathrm{syz}^1_RA$ and $g_i^j$ are defined iteratively as:

\begin{center}
    $g_i^j = \frac{\gen{F_{T}^{T_j},zH_{T}^{T_i} - g_i^1F_{T'}^{T_1'} - \hspace{0.1em}\cdots\hspace{0.1em} - g_i^{j-1}F_{T'}^{T_{j-1}}}}{\gen{F_{T}^{T_j}, F_{T'}^{T_j'}}}$
\end{center}
for $0\leq i \leq m$ and $0 \leq j \leq m$.
\end{thm}

\begin{proof}
Note that $\gen{F_T^S,F_{T'}^{S'}}$ is a non-zero constant in $k$, and a formula is given in \cite[Proposition 1]{ariki1997}. Order $\ST(\lm) = (S_1,\ldots,S_m)$ where if $i<j$ then $S_i<S_j$, so if $i<j$ then $\gen{F_T^{S_i},F_{T'}^{S_j'}} = 0$. 
Consider the matrix describing  $z|_{M_T}: M_{T} \to N_{T'}$ to have the entries $[g_i^j]_{i,j}$ where $i,j$ indexes the rows and columns. 
We calculate
\[zH_{T}^{S_i} = g_i^1 F_{T'}^{S_1'} + \cdots + g_i^m F_{T'}^{S_m'} \]

Plugging this into the bilinear form with $F^{S_j}_{T}$ yields
\begin{equation}\label{matrixmult}
\gen{F^{S_j}_{T},zH_T^{S_i}}= g^1_i\gen{F^{S_j}_{T},F^{S_1'}_{T'}} + \dots + g^m_i\gen{F^{S_j}_{T},F^{S_m'}_{T'}} \ . 
\end{equation}
For $1<j$ the term $g^j_i\gen{F^{S_1}_{T},F^{S_j'}_{T'}}=0$, therefore $\gen{F_{T}^{S_1},zH_T^{S_i}} = g_i^1\gen{F_{T}^{S_1}, F_{T'}^{S_1'} }$. Thus we can then recursively write a formula for each entry.

\[
    g_i^1 = \frac{\gen{F_{T}^{S_1},zH_{T}^{S_i}}}{\gen{F_{T}^{S_1}, F_{T'}^{S_1'}}} \]
    \[
    \vdots\]
\[
    g_i^j = \frac{\gen{F_{T}^{S_j},zH_{T}^{S_i} - g_i^1F_{T'}^{S_1'} - \hspace{0.1em}\cdots\hspace{0.1em} - g_i^{j-1}F_{T'}^{S_{j-1}}}}{\gen{F_{T}^{S_j}, F_{T'}^{S_j'}}}
    \]

These are the entries in row $i$, and the matrix $A$ can be computed by considering all $i$.
\end{proof} 
\begin{remark}
This gives a quicker computational way to calculate the matrix factorization corresponding to a irreducible representation of $S_n$, and thus a maximal Cohen--Macaulay module over $R$, for a specific irreducible representation $\lambda$ of $S_n$.
\end{remark}

\begin{remark}
Recall that if $L$ is an lower triangular $k \times k$  matrix, we can write $L= D+N$
where $D$ is diagonal and $N$ is strictly lower triangular.
If $D$ is invertible, then $D^{-1}N$ is nilpotent and we have the well known formula 
\begin{eqnarray*}
L^{-1} & = & D^{-1} - D^{-1}ND^{-1} + D^{-1}ND^{-1}ND^{-1}  - \cdots \\
& = & D^{-1}\left(\sum_{i=0}^n (ND^{-1})^i\right). 
\end{eqnarray*}
If we define matrices
\begin{eqnarray*}
U_{ij} & =& \gen{F_{T}^{S_j}, F_{T'}^{S_i'}}\\
G_{ij} & = &g^j_i\\
X_{ij} & = &\gen{F_{T}^{S_j},zH_{T}^{S_i}}
\end{eqnarray*}
then Equation~\eqref{matrixmult} gives $X = GL$ and Lemma~\ref{Lem:Bilinear} shows that $L$ is lower triangular.  So we can solve the recursive formula 
in Theorem~\ref{Bilinear} as 
\begin{eqnarray*}
G & = & X L^{-1} \\
& = & X(D^{-1} - D^{-1}ND^{-1} + D^{-1}ND^{-1}ND^{-1}  - \cdots) \\
& = & XD^{-1}\left(\sum_{i=0}^n (ND^{-1})^i\right). 
\end{eqnarray*}
\end{remark}

\begin{example} \label{Ex:MFS5}
Let $S_5$ act on $\mathbb{C}[x_1,x_2,x_3,x_4,x_5]$ with the basic invariants $e_i$, $i=1, \ldots, 5$. If we quotient out by the hyperplane $e_1=x_1+ \cdots + x_5=0$, we get a set of invariants $t_1,\dots,t_4$ of the action of $S_5$ on $k[x_2,x_3,x_4,x_5]$, where $t_i=e_{i+1}(-x_2-x_3-x_4-x_5,x_2,x_3,x_4,x_5)$. The discriminant $\Delta$ of this group action is given by: 

\begin{equation*} \begin{split}
    \Delta=& -\frac{1}{3\,600}\,t_{1}^{3}t_{2}^{2}t_{3}^{2}+\frac{1}{900}\,t_{1}^{4}t_{3}^{3}+\frac{1}{900}\,t_{1}^{3}t_{2}^{3}t_{4}-\frac{1}{200}\,t_{1}^{4}t_{2}t_{3}t_{4}+\frac{3}{400}\,t_{1}^{5}t_{4}^{2}-\\ &\frac{3}{1\,600}\,t_{2}^{4}t_{3}^{2}  +\frac{1}{100}\,t_{1}t_{2}^{2}t_{3}^{3}-\frac{2}{225}\,t_{1}^{2}t_{3}^{4}+\frac{3}{400}\,t_{2}^{5}t_{4}-\frac{7}{160}\,t_{1}t_{2}^{3}t_{3}t_{4}+\\ &\frac{7}{180}\,t_{1}^{2}t_{2}t_{3}^{2}t_{4}+\frac{11}{192}\,t_{1}^{2}t_{2}^{2}t_{4}^{2}-\frac{1}{16}\,t_{1}^{3}t_{3}t_{4}^{2}+\frac{4}{225}\,t_{3}^{5}-\frac{1}{9}\,t_{2}t_{3}^{3}t_{4}+\frac{5}{32}\,t_{2}^{2}t_{3}t_{4}^{2}+\\&\frac{5}{36}  \,t_{1}t_{3}^{2}t_{4}^{2}-\frac{25}{96}\,t_{1}t_{2}t_{4}^{3}+\frac{125}{576}\,t_{4}^{4}.\end{split} \end{equation*}

Let $\lambda={\tiny\ydiagram{4,1}}$ be a partition of $n$ corresponding to the standard representation of $S_n$, then the matrix factorization of $S/(z)$ corresponding to $\lambda$ is $(A,B)$ where; 
\[A=\left(\!\begin{array}{cccc}
t_{4}&-\frac{1}{50}\,t_{1}t_{3}&-\frac{1}{50}\,t_{2}t_{3}+\frac{1}{10}\,t_{1}t_{4}&-\frac{1}{25}\,t_{3}^{2}+\frac{1}{10}\,t_{2}t_{4}\\
-\frac{8}{5}\,t_{3}&\frac{2}{25}\,t_{1}t_{2}-\frac{1}{2}\,t_{4}&\frac{2}{25}\,t_{2}^{2}-\frac{2}{15}\,t_{1}t_{3}&\frac{3}{50}\,t_{2}t_{3}-\frac{3}{10}\,t_{1}t_{4}\\
\frac{6}{5}\,t_{2}&-\frac{3}{25}\,t_{1}^{2}+\frac{2}{5}\,t_{3}&-\frac{4}{75}\,t_{1}t_{2}+\frac{1}{3}\,t_{4}&-\frac{1}{25}\,t_{1}t_{3}\\
-\frac{4}{5}\,t_{1}&-\frac{3}{10}\,t_{2}&-\frac{4}{15}\,t_{3}&-\frac{1}{2}\,t_{4}\\
\end{array}\!\right)\]

$B$ is a $4 \times 4 $ matrix with entries in $R$ such that $\cok B \cong \text{Syz}^1_R(A)$. Note that as the dimension of the representation corresponding to $\lambda$ is $4$, we get $4$ copies of this matrix.

\subsection{Computation of the matrices with Macaulay2} \label{Sec:M2}
Both of the matrices from Example \ref{Ex:MFS5} were obtained using the Macaulay2 package ``PushForward'' \cite{PushForwardSource}. Consider a polynomial ring $S$ and a subring $R$, such that $S$ is a free and finite module over $R$. This package serves to calculate a monomial basis of $R$ over $S$, such that one can write elements of $R$ as vectors over this basis. Furthermore for given $S$-modules $M$ and $N$, with a module map $\phi: M \to N$, this package also writes $\phi$ as a matrix over $R$.\\

By modifying the function which computes the basis in this package we are able to choose any basis for the ring $S$, as described above, over $R$. This can be done simply by computing the change of basis matrix and its inverse, where we change the monomial basis computed by Macaulay2 to any basis we need. Thus we can write any $R$-module map using the basis we need. \\

In the setting of Example \ref{Ex:MFS5}, consider $S = \C[x_1,\hdots,x_5]$ and consider the generators of $R=S^{S_5}$ to be the elementary symmetric functions $\{e_1,\hdots,e_5\}$. Let the map $\phi: S \to S$ be given by $\phi: x \mapsto zx$. This way $\phi$ is an $R$-module map from $S$ to $S$. Next we define the following two bases:

\[ 
B_H = \bigcup_{\lambda \vdash 5}\{ H_T^V \st T,V \in \ST(\lambda) \} \text{ and } B_F =  \bigcup_{\lambda \vdash 5}\{ F_T^V \st T,V \in \ST(\lambda) \} \ ,
\]
where $H_T^V$, $F_T^V$ are the polynomials from Definition \ref{def:HSP}.
Here we can assume the basis is ordered by the partitions, following the lexicographical ordering.

\vspace{-1em}
\[ (5) < (4,1) < (3,2) < (3,1,1) < (3,2)' < (4,1)' < (5)' \]

We denote the first four partitions in the following way $ \lambda_1 < \lambda_2 < \lambda_3 < \lambda_4$. Furthermore denote $H_{\lambda_i} = \{H_T^V \st T,V \in \ST(\lambda_i) \}$ and $F_{\lambda_i} = \{ F_T^V \st T,V \in \ST(\lambda_i) \}$. Thus we write the following:

\[
B_H = H_{\lambda_1}\cup H_{\lambda_2}\cup H_{\lambda_3}\cup H_{\lambda_4}\cup H_{\lambda_3'}\cup H_{\lambda_2'}\cup H_{\lambda_1'} \ .
\]

We can apply a similar construction to the $B_F$ basis. Using these bases, we can use the Macaulay2 package ``PushForward'' with the modification to compute the $R$ module map $\phi:R^{5!} \to R^{5!}$ as a matrix using the basis $B_H$ on the domain and the basis $B_F$ on the codomain. Each one of the $H_{\lambda_i}$ forms a basis for the isotypical component associated with the partition $\lambda_i$, we denote this isotypical component as $S_{\lambda_i}$. Because the map $\phi$ is an anti-linear map, the image of a isotypical component under this map yields $\phi(S_\lambda) \subset S_{\lambda'}$. Therefore we obtain the following matrix as a push forward of the map $\phi$ using the above basis:

\[
\phi = 
\begin{blockarray}{cccccccc}
 & H_{\lambda_1} & H_{\lambda_2} & H_{\lambda_3} & H_{\lambda_4} & H_{\lambda_3'} & H_{\lambda_2'} &
 H_{\lambda_1'}\\
\begin{block}{c(ccccccc)}
  F_{\lambda_1}  & 0 & 0 & 0 & 0 & 0 & 0 & M_6 \\
  F_{\lambda_2}  & 0 & 0 & 0 & 0 & 0 & M_5 & 0\\
  F_{\lambda_3}  & 0 & 0 & 0 & 0 & M_4 & 0 & 0 \\
  F_{\lambda_4}  & 0 & 0 & 0 & M_4 & 0 & 0 & 0 \\
  F_{\lambda_3'} & 0 & 0 & M_3 & 0 & 0 & 0 & 0\\
  F_{\lambda_2'} & 0 & M_2 & 0 & 0 & 0 & 0 & 0\\
  F_{\lambda_1'} & M_1 & 0 & 0 & 0 & 0 & 0 & 0\\
\end{block}
\end{blockarray}
\]

Here we have a block decomposition of this matrix. Furthermore $(\phi,\phi)$ is a matrix factorization of the $S_n$ discriminant $\Delta$. With this reasoning, for each $i$ from $1$ to $6$ we have that $M_iM_{6-i} = \Delta I$ where $I$ is the identity matrix of size ${\dim(V_{\lambda_i})}$. In order to obtain our matrix $A$ we will further decompose each block. For each basis $H_\lambda$ and $F_\lambda$ associated with the partition we will further order it in the following way

\[ 
H_\lambda = \bigcup_{T \in \ST(\lambda)} \{ H_T^V \st V \in  \ST(\lambda) \} \text{ and }
F_\lambda = \bigcup_{T \in \ST(\lambda)} \{ F_T^V \st V \in  \ST(\lambda) \} \ .
\]

After this decomposition, via Theorem \ref{Big} and Theorem \ref{Big-equiv} we decompose $M_2$ into exactly $3$ copies of the matrix $A$ from Example \ref{Ex:MFS5}, since $\abs{\ST((4,1))} = 3$. Thus using our Macaulay2 code we can compute all matrices for $S_n$ in this manner. Our packages can be found at \cite{Pack}. Below, in Fig.~\ref{fig:S_N32}, we also share the Matrix associated with the partition $(3,2)$.
\end{example}

\vspace{1em}

\begin{example}
\ytableausetup{smalltableaux}
Let $\lambda={\tiny\ydiagram{3,2}}$ then the matrix factorization of $S/(z)$ corresponding to $\lambda$ is $(A,B)$ where $A$ is the matrix given in Figure \ref{fig:S_N32} and $B$ is a $5\times 5$ matrix with entries in $R$. 
\end{example} 
\ytableausetup{nosmalltableaux}

\begin{figure}[h]
\vspace{1em}
\begin{center}
$\left[ \begin{smallmatrix}
-6t{_1}^{2}t{_3}+24t{_3}^{2}-180t{_2}t{_4} & 3t{_2}^{2}t{_3}-23t{_1}t{_2}t{_4}-25t{_4}^{2} & -2t{_1}t{_2}t{_3}+12t{_1}^{2}t{_4}-20t{_3}t{_4} \\ 
-36t{_2}t{_3}+270t{_1}t{_4}  & -4t{_1}t{_2}t{_3}+36t{_1}^{2}t{_4}+40t{_3}t{_4} & -32t{_3}^{2}+90t{_2}t{_4}  \\
12t{_1}t{_2}+600t{_4} & -16t{_2}t{_3}+120t{_1}t{_4} & 12t{_2}^{2}-32t{_1}t{_3} \\
144t{_2}^{2}-384t{_1}t{_3}&16t{_1}t{_2}^{2}-48t{_1}^{2}t{_3}-64t{_3}^{2}+80t{_2}t{_4}&32t{_2}t{_3}-240t{_1}t{_4} \\ 
36t{_1}^{2}+240t{_3}&-24t{_2}^{2}+64t{_1}t{_3}&4t{_1}t{_2}+200t{_4} 
\end{smallmatrix} \right.$\\
\vspace{0.5em}
\hspace{10em}$\left. \begin{smallmatrix}
6t{_2}t{_3}^{2}-18t{_2}^{2}t{_4}+3t{_1}t{_3}t{_4}&2t{_1}^{2}t{_2}t{_4}+12t{_2}t{_3}t{_4}+10t{_1}t{_4}^{2} \\
-8t{_1}t{_3}^{2}+24t{_1}t{_2}t{_4}+75t{_4}^{2}&18t{_2}^{2}t{_4}-48t{_1}t{_3}t{_4} \\
-32t{_3}^{2}+90t{_2}t{_4}&-12t{_1}^{2}t{_4}-80t{_3}t{_4} \\
8t{_1}t{_2}t{_3}-72t{_1}^{2}t{_4}-80t{_3}t{_4}&-18t{_2}^{2}t{_3}+64t{_1}
t{_3}^{2}-46t{_1}t{_2}t{_4}-50t{_4}^{2} \\
-12t{_2}t{_3}+90t{_1}t{_4}&-4t{_1}^{2}t{_3}-48t{_3}^{2}+60t{_2}t_{4}
\end{smallmatrix}\right]$
\caption{The Matrix factorization for the partition $(3,2) \vdash 5$}\label{fig:S_N32}
\end{center}
\end{figure}

\section{Decomposition for product submodules of $S_n$} \label{Sec:YoungSubGroups}

In this section we generalize Theorem \ref{Big} to irreducible representations of the \emph{Young Subgroups} of $S_n$. These subgroups are of the form $S_{n_1}\times \cdots \times S_{n_m} \leq S_n$ for any given $m$-tuple $(n_1,\hdots,n_m)$ with $\sum_{i=1}^m n_i = n$. In particular the decomposition of $(z,z)$ will correspond to the irreducible representations of $S_{n_1}\times \cdots \times S_{n_m}$. The irreducible representations of the Young subgroup $S_{n_1}\times \cdots \times S_{n_m}$  will be of the form $V_{n_1}\otimes \hdots \otimes V_{n_m} $ where each $V_{n_i}$ is a irreducible representation of $S_{n_i}$ thus we will discuss a basis for the coinvariant algebra $S/(R_+)$ indexed by these representations. While the decomposition will be more coarse than the one discussed in Section \ref{Sec:Decomp}, the motivation for this section is that the construction can be used to describe the decomposition for the wreath product groups $G(m,1,n)$ from the Shephard--Todd classification. 
\\
The definitions from Section \ref{YD} can be generalized to describe the representations of Young subgroups. Consider $m \geq 0$ and $(n_1,\hdots,n_m)$ to be a list of integers such that $n_i\geq0$ and $\sum n_i = n$. Let $\lambda= (\lambda_1,\hdots,\lambda_m)$ be an $m$-tuple of Young diagrams of type $(n_1,\ldots,n_m)$ if $\lambda_i \vdash n_i$ for all $1\leq i \leq m$. An $m$-tuple of tableaux $T = (T_1,\hdots,T_m)$ is of shape $\lambda$ if each $T_i$ is of shape $\lambda_i$, and is called an $m$\emph{-tableau}. An $m$-tableau is standard if all of its tableaux are standard, with the set of all standard $m$-tableau being $\ST(\lambda)$. We define $\ST(n_1,\hdots,n_m)$ to be the set of standard $m$-tableau of type $(n_1,\hdots,n_m)$. If $T=(T^1,\dots,T^m)$ is a standard $m$-tableau of type $(n_1,\hdots,n_m)$ then $T'=((T^{1})',\dots,(T^m)').$ 
\begin{example}

Let $n=7$, $m = 3$ then 
\[ T=
\left(
\begin{ytableau}
1 & 7\\
5
\end{ytableau}\mathrel{},
\mathrel{}
-\mathrel{},
\mathrel{}
\begin{ytableau}
2&3\\
4&6
\end{ytableau}
\right)
\]

is a standard $m$-tableau of type $(3,0,4)$. The conjugate tableau $T'$ is given by

\[T'=\left(
\begin{ytableau}
1 & 5\\
7
\end{ytableau}\mathrel{},
\mathrel{}
-\mathrel{},
\mathrel{}
\begin{ytableau}
2&4\\
3&6
\end{ytableau}
\right)\]
\end{example} 

\begin{remark}
We have defined $T'$ differently to \cite{ariki1997}, where they also reverse the order of the tableau, this is so that,  given a $T \in \ST(\lambda)$, the $m$-tableau $T'$ is in $\ST(\lambda')= \ST(\lambda\otimes \det)$. The consequence of our definition is that the we will not be able to use the same bilinear form reduction to get a similar result to Theorem \ref{Bilinear} as before. 

\end{remark}
Similarly as in Section \ref{Sec:Decomp} we will define an ordering on $\ST(n_1,\hdots,n_m)$ as an extension of the Last Letter ordering. First we will consider an ordering on $\ST(\lambda)$. Consider two standard $m$-tableaux $T_1=(T_1^1,\hdots,T_1^m)$ and $T_2=(T_2^1,\hdots,T_2^m)$ of the same shape. Let $1 \leq k \leq n$ be the greatest number that appears in different cells in both of the tableaux. We say $T_1 < T_2$ if either $k$ is written in $T_1^i$ and $T_2^j$ with $i < j$, or it is written in a row in $T_1^i$ below a row in $T_2^i$, for $1 \leqslant i \leqslant n$.\\
\\
\begin{example}
\ytableausetup{smalltableaux}
Let $n=4$, $m=2$ and $\lambda= \left( \ydiagram{2,1},\ydiagram{1}\right)$ then, with respect to the Last Letter ordering on ST($\lambda$):
\ytableausetup{nosmalltableaux}
\vspace{1em}
\[
\left(
\begin{ytableau}
1&2\\
4
\end{ytableau}\mathrel{},
\mathrel{}
\begin{ytableau}
3
\end{ytableau}
\right)<
\left(
\begin{ytableau}
1&2\\
3
\end{ytableau}\mathrel{},
\mathrel{}
\begin{ytableau}
4
\end{ytableau}
\right)
\]
\ytableausetup{smalltableaux}

Now we consider when the last number that appears in the different cells is on the same tableaux for both $m$-tableaux.
Let $n=4, m=2$ and $\lambda= \left( \ydiagram{2,2},\ydiagram{2,1}\right)$ then with respect to the Last Letter ordering on ST($\lambda$):
\ytableausetup{nosmalltableaux}
\vspace{1em}
\[
\left(
\begin{ytableau}
1&2\\
3&4
\end{ytableau}\mathrel{},
\mathrel{}
\begin{ytableau}
5&6\\
7
\end{ytableau}
\right)<
\left(
\begin{ytableau}
1&2\\
3&4
\end{ytableau}\mathrel{},
\mathrel{}
\begin{ytableau}
5&7\\
6
\end{ytableau}
\right)
\]
\end{example} 

\begin{thm}
Let $\lambda$  be a Young diagram of type $(n_1,\hdots,n_m)$, then the last letter ordering is total in $\ST(\lambda)$.
\end{thm}

\begin{proof}
Let $T_1,T_2 \in \ST(\lambda)$ then either $T_1=T_2$ or there exists, at least 2 elements which appear in different boxes. Let $k$ be the last number that changes, then $k$ must appear in a different rows otherwise one of $T_1,T_2$ would not be standard. Since the last number that changes appears in different rows then either $T_1<T_2$ or $T_2<T_1$.
\end{proof}

If we consider the lexicographical ordering of the partitions of type $(n_1,\hdots,n_m)$ we can fully order $\ST(n_1,\hdots,n_m)$. Let us consider a ordering on partitions of $n$, if $\lambda_1 = (\alpha_1, \hdots, \alpha_k)$ and $\lambda_2 = (\beta_1,\hdots,\beta_l)$ are two partitions of $n$, and let $1\leq i \leq \min(k,l)$ be the first integer that $\alpha_i - \beta_i \neq 0$, then if $\alpha_i - \beta_i > 0$ then $\lambda_1 < \lambda_2$. Using this ordering, it is easy to see that we can use lexicographical ordering on partition of type $(n_1,\hdots,n_m)$ to have a total ordering. This way we can order tableaux of different partitions by comparing their shapes. This way if we have two tableaux $T$ and $V$, if they are in the same partition we may order them using LL-order, and if they are in different partition give their order with lexicographical order on the partitions.

\begin{example}
Let us consider $m=1$ and $n =4$ then we can order the partitions the following way;
\[ (4) < (3,1) < (2,2) < (2,1,1) < (1,1,1,1)\]
If we consider $\lambda_1 = (3,1)$ and $\lambda_2 =(2,2)$ and $T \in \ST(\lambda_1)$ and $V \in \ST(\lambda_2)$ as below, by our ordering we have:
\[
\begin{ytableau}
1&2&3\\ 4
\end{ytableau} <
\begin{ytableau}
1&2\\ 3&4
\end{ytableau}
\]
\\
\ytableausetup{smalltableaux}

Now consider $m=3$ and $n=6$, and define  $\lambda_1 = \left(\ydiagram{2,1},\ydiagram{1},\ydiagram{2}\right)$ and $\lambda_2 = \left(\ydiagram{2,1},\ydiagram{1},\ydiagram{1,1} \right)$ then since $\ydiagram{2} < \ydiagram{1,1}$ we have the following:
\ytableausetup{nosmalltableaux}
\[
\left(
\begin{ytableau}
2&3\\ 5
\end{ytableau},
\begin{ytableau}
1
\end{ytableau},
\begin{ytableau}
4&6\\
\end{ytableau} \right) 
<
\left(
\begin{ytableau}
2&3\\ 5
\end{ytableau},
\begin{ytableau}
1
\end{ytableau},
\begin{ytableau}
4\\
6
\end{ytableau} \right) 
\]
\end{example}

\begin{defn}
\sloppy{Let $\lambda$ be an $m$-partition and consider an $m$-tableau $T= (T_1,\hdots,T_m)$ of shape $\lambda$. We call $T$ \emph{natural} if the numbers written in tableau $T_i$ are contained in the set $\{\sum_{j=1}^{i-1}n_j+1,\hdots,\sum_{j=1}^{i}n_j\}$. We denote the set of natural standard $m$-tableaux of shape $\lambda$ by $\NST(\lambda)$ and $\NST(n_1,\hdots,n_m)$ is the set of all natural standard tableaux on all the partitions of type $(n_1,\ldots,n_m)$}.\end{defn}

\ytableausetup{smalltableaux}

\begin{example}
Take $m=3$ and $n = 5$, and consider the shape 
$\lambda = \left(\ydiagram{2,1},-,\ydiagram{2}\right)$. To make a natural standard tableaux with this shape we take take the first three $\{1,2,3\}$ and assign to the first tableaux, we assign no numerals to the second tableaux since it is empty, and the remaining numbers go in the last tableaux. Then we have to sort the numerals in the tableaux in order to make them standard. Thus
\ytableausetup{nosmalltableaux}
\[
\left(
\begin{ytableau}
1&3\\ 2
\end{ytableau},
-,
\begin{ytableau}
4&5
\end{ytableau}
\right)
\]
is an element of $\NST(\lambda)$. 
\end{example}

We can define Young symmetrizers in a similar fashion to the previous section, for a given $m$-tableaux $T = (T_1,\hdots,T_m)$ we define $\e_T = \e_{T_1}\cdots\e_{T_m}$ and similarly $\s_T = \s_{T_1}\cdots\s_{T_m}$. The following theorem shows how $S/(R_+)$ decomposes into irreducible representations of a Young subgroup.
\\

\begin{thm}\cite[Theorem 1]{ariki1997} 
Fix $n$ and let $(n_1,\ldots,n_m)$ be a sequence such that $\sum_{i=1}^m n_i =n$. Then the collection:
\[\bigcup_{\lambda\vdash  (n_1,...,n_{m})} \{F^{S}_{T} \,|\, T \in \NST(\lambda), S \in \ST(\lambda)\}.\] 
Form a $k$-basis for $S/(R_+)$. For $\lambda \vdash (n_1,\ldots,n_{m})$, let $S \in \ST(\lambda)$. Then the collection
\[\{F^S_T \,|\, T \in \NST(\lambda)\}\]
forms a basis of the $S_{n_1}\times \cdots \times S_{n_{m}}$-submodule of $S/(R_+)$ which is isomorphic to irreducible representation $V_\lambda$.
\end{thm}

\begin{example}
Let us consider the Young subgroup $S_1 \times S_2$ inside $S_3$. There are two, $2$-tuples of young diagrams that partition (1,2) namely: 

\[\lambda_1 =\left(\ydiagram{1},\ydiagram{2}\right) \hspace{2cm} \lambda_2 =\left( \ydiagram{1},\ydiagram{1,1}\right)\]

\[\NST(\lambda_1)=\left\{\left(\begin{ytableau}
1
\end{ytableau},\begin{ytableau}
2 & 3
\end{ytableau}\right)\right\} \hspace{2cm} \NST(\lambda_2)=\left\{\left(\begin{ytableau}
1
\end{ytableau},\begin{ytableau}
2 \\ 3
\end{ytableau}\right) \right\}
\]
and 

\[\ST(\lambda_1)=\left\{\left(\begin{ytableau}
1
\end{ytableau},\begin{ytableau}
2&3
\end{ytableau}\right), \left(\begin{ytableau}
2
\end{ytableau},\begin{ytableau}
1&3
\end{ytableau}\right), \left(\begin{ytableau}
3
\end{ytableau},\begin{ytableau}
1&2
\end{ytableau}\right)\right\}
\]
\[\ST(\lambda_2)=\left\{\left(\begin{ytableau}
1
\end{ytableau},\begin{ytableau}
2\\3
\end{ytableau}\right), \left(\begin{ytableau}
2
\end{ytableau},\begin{ytableau}
1\\3
\end{ytableau}\right), \left(\begin{ytableau}
3
\end{ytableau},\begin{ytableau}
1\\2
\end{ytableau}\right)\right\}
\]
\end{example}

Thus we get $3$ copies of the irreducible representation corresponding to $\lambda_1$ and $3$ copies of the irreducible representation corresponding to $\lambda_2$
\begin{lem}
Let $T_1,T_2 \in$ NST$(n_1,\ldots,n_m)$. If $T_1 < T_2 $, it follows that $\varepsilon_{T_1}\varepsilon_{T_2}=0$.
\end{lem}
\begin{proof}
Let $T_1,T_2\in\NST(n_1,\ldots,n_m)$, then we can commute terms of the Young symmetrizer such that $\e_{T_1}\e_{T_2}=\e_{T_1^1}\cdots \e_{T_1^m}\e_{T_2^1}\cdots\e_{T_2^m} = \e_{T_1^1}\e_{T_2^1}\cdots\e_{T_1^m}\e_{T_2^m}$  Let $T_1<T_2$ and suppose that the last number that appears in different cells is contained in $T_1^{i}$. If $T_1$ and $T_2$ are of the same shape then $\varepsilon_{T_1^{i}}\varepsilon_{T_2^{j}}=\varepsilon_{T_2^{j}}\varepsilon_{T_1^{i}}$ for $i\neq j$. Now $T_1^{i}<T_2^{i}$, then $\varepsilon_{T_1}\varepsilon_{T_2}=0$ follows from Lemma \ref{Zeroing}. If $T_1^{i}$ and $T_2^{i}$ are of different shapes, $
\varepsilon_{T_1^{i}}
\varepsilon_{T_2^{i}}=0$ and thus $\varepsilon_{T_1}\varepsilon_{T_2}=0$.
\end{proof}

\begin{lem}\label{Lem:commute}
Let $T$ be a standard $m$-tuple of tableau, then $z\varepsilon_T(f)=\sigma_{T'}(f)$ for any $f\in S$. 
\end{lem}

\begin{proof}
Let $f \in S$ then
$z\varepsilon_T(f)=z\varepsilon_{T^1}\dots\varepsilon_{T^m}(f)=\varepsilon_{T^1}\dots\varepsilon_{T^{m-1}}z\sigma_{(T^m)'}(f)$  $=\sigma_{(T^1)'}\dots\sigma_{(T^m)'}(zf)$.
\end{proof}
\begin{defn}
Let $\lambda$ be an $m$-tuple of Young diagrams of size $(n_1,\dots,n_m)$ and let $T \in \NST(\lambda)$. Let $M_T= \langle H^S_T \,|\, S \in \ST(\lambda) \rangle$ and $N_T=\langle\{F_{T}^S \,|\, S \in \ST(\lambda') \}\rangle$ be $R$-modules.
\end{defn}

\begin{remark}
These are analogous modules to the ones defined in Definition \ref{def:freemods} and are free $R$-submodules of $S/(R_+)$, but are not irreducible representations of $S_n$ .
\end{remark} 

\begin{thm}\label{Prodsub} For the discriminant $\Delta$ of $S_n$, the matrix factorization defined by the reduced hyperplane arrangement, $(z,z)$, can be decomposed in the following way:

\[(z,z) = \bigoplus_{\lambda \vdash (n_1,\ldots, n_m)} \bigoplus_{T\in \NST(\lambda)} (z|_{M_{T}},z|_{N_{T'}}).\]
 and $(z|_{M_{T}},z|_{N_{T'}})$ are the matrix factorizations:

  \begin{center}
    \begin{tikzcd}
   M_{T} \arrow[r, "z|_{M_{T}}"] & N_{T'} \arrow[r, "z|_{N_{T'}}"] & M_{T}
    \end{tikzcd}
    \end{center}

\end{thm}

\begin{proof}
\sloppy{Recall that we can order all of the standard tableaux with $n$ cells $\NST (n_1,\dots,n_m) $ such that if $i < j$ then} $T_i>T_j$, thus $\e_{T_j}\e_{T_i} = \s_{T_i}\s_{T_j} = 0$. Let $d$ to be the size of  $\NST(n_1,\ldots,n_m)$, we can write $S = \oplus_{1\leq i\leq d} M_{T_i} = \oplus_{1\leq i\leq d} N_{T_i}$. It is clear that since for a $m$-tableau $T$ $\varepsilon_{T}$ and $\sigma_{T}$ are idempotent  if $i<j$ then $\e_{T_i}F_{T_j} = 0$ and $\s_{T_j}H_{T_i} = 0$ . Therefore consider $1\leq k\leq d$, let $P$ be a standard tableaux of the same shape as $T_k$. Then we can split $zH_{T_k'}^P$ into the different components of $S$, where each $f_{T_i'} \in N_{T_i'}$. Calculate
\begin{align}\label{predec}
    zH_{T_k}^P = f_{T_1'} + f_{T_2'} + \cdots + f_{T_k'} + \cdots + f_{T_d'} \ . 
\end{align}
Claim: For each $1 \leq j < k$, each component $f_{T_j'} = 0$.

We prove the claim by induction. Let $j=1$, then since $1<k$ then $T_k<T_1$ thus $\s_{T_1}H_{T_k}^P = 0$.
    \[\e_{T_1'}zH_{T_k}^P = \e_{T_1'}(f_{T_1'} +\cdots +f_{T_d'})\]
    From Lemma \ref{Lem:commute} we have
   \begin{align*} z(\s_{T_1}H_{T_k}^P) &= \e_{T_1'}f_{T_1'} + \cdots +\e_{T_1'}f_{T_d'}\\
    0 &= f_{T_1'}
\end{align*}
Assume that the claim is true for $j-1$. Since $j<k$ then $T_j > T_k$ thus $\s_{T_j}H_T^P = 0$. Therefore we have the following computation, using again Lemma \ref{Lem:commute}
\begin{align*}
    \e_{T_j'}zH_{T_k}^P &= \e_{T_j'}(f_{T_1'} + \cdots + f_{T_j'}+\cdots +f_{T_d'})\\
    z(\s_{T_j}H_{T_k}^P) &= \e_{T_j'}f_{T_1'} + \cdots + \e_{T_j'}f_{T_j'}+\cdots +\e_{T_j'}f_{T_d'}\\
    0 &= f_{T_j'}
\end{align*}
Therefore equation \eqref{predec} reduces to
\begin{align} \label{eq:3}
    zH_{T_k}^S = f_{T_k'} + f_{T_{k+1}'} + \cdots + f_{T_d'} .
\end{align}
 After applying $\e_{T_k'}$ to \eqref{eq:3}, the left hand side becomes $\e_{T_k'}(zH_{T_{k}}^P) = z(\s_{T_k}H_{T_k}^P) = zH_{T_k}^P$. If $j>k$, then $\e_{T_k'}f_{T_j'} = 0$, thus $zH_{T_k}^K = f_{T_k'}$. In other words $z|_{M_{T}}:M_{T} \to N_{T'}$. A similar argument can be made about $zF_{T_i}^P$ thus proving the statement. \\
\\
The argument above shows that for a $T \in \NST(n_1,...,n_m)$, $\text{Im}(z|_{M_T}) = N_{T'}$, similarly one could show that $\text{Im}(z|_{N_{T'}}) =M_T$. Therefore the matrix factorization splits as

\[(z,z) = \oplus_{T \in NST(n_1,\ldots,n_m)}(z|_{H_T}, z|_{F_{T'}}).\]
\end{proof}

\begin{example}

Consider the Young subgroup $S_1\times S_2$ inside $S_3$ and let $\sigma_1,\sigma_2,\sigma_3$ be the invariants of the $S_3$ action, then the matrix representing the multiplication by $z$ is given by;

\[\begin{bmatrix}
0 & B \\ A & 0
\end{bmatrix}\]
where $A$ is the $3\times 3$ matrix:

\[A = \begin{bmatrix}
-2\sigma_{2} & -2\sigma_{1}\sigma_{3} &-6\sigma_{3}\\ 
4\sigma_{1} & \sigma_{1}\sigma_{2}+3\sigma_{3}& 4\sigma_{2} \\
-6 &-2\sigma_{2} &-  2\sigma_{1}
\end{bmatrix}\]

$B$ is the $3 \times 3$ matrix:

\[B = \begin{bmatrix}
-\frac{1}{2}\sigma_{1}^{2}\sigma_{2}+2\sigma_{2}^{2}-\frac{3}{2}\sigma_{1}\sigma_{3} &-\sigma_{1}^{2}\sigma_{3}+3\sigma_{2}\sigma_{3}&-\frac{1}{2}\sigma_{1}\sigma_{2}\sigma_{3}+\frac{9}{2}\,\sigma_{3}^{2}\\ 
2\sigma_{1}^{2}-6\sigma_{2}&\sigma_{1}\sigma_{2}-9\sigma_{3}&2\sigma_{2}^{2}-6\sigma_{1}\sigma_{3} \\
-\frac{1}{2}\sigma_{1}\sigma_{2}+\frac{9}{2}\sigma_{3}& -\sigma_{2}^{2}+3\sigma_{1}\sigma_{3} &-\frac{1}{2}\,\sigma_{1}\sigma_{2}^{2}+2\sigma_{1}^{2}\sigma_{3}-\frac{3}{2}\sigma_{2}\sigma_{3}
\end{bmatrix}\]

$A$ is the matrix of $(z|_{H_T},z|_{F_{T'}})$ where;

\[ T= \left(\begin{ytableau}
1
\end{ytableau},\begin{ytableau}
2 & 3
\end{ytableau}\right)
\hspace{2cm}\left(\begin{ytableau}
1
\end{ytableau},\begin{ytableau}
2 \\ 3
\end{ytableau}\right)\] of $\NST(1,2)$.
\end{example}

\begin{remark}
The matrix factorizations from Theorem \ref{Big} and \ref{Prodsub} are equivalent as matrix factorizations, since they are matrices that describe the same map. We get from one to the other by a change of basis of $S/(R_+)$.
\end{remark}

\def\cprime{$'$}

\end{document}